\documentclass[12pt]{amsart}
\usepackage{amssymb,amsmath,amsfonts,latexsym}
\usepackage{bm}
\usepackage{array,graphics,color}
\newcommand{\newsection}[1]{\setcounter{equation}{0} \section{#1}}
\numberwithin{equation}{section}

\newtheorem{propn}{Proposition}[section]
\newtheorem{thm}[propn]{Theorem}
\newtheorem{lemma}[propn]{Lemma}
\newtheorem{cor}[propn]{Corollary}
\newtheorem{remark}[propn]{Remark}

\newtheorem*{thm*}{Theorem}
\theoremstyle{definition}
\newtheorem{defn}[propn]{Definition}

\newcommand{\Z}{\mathbb{Z}_+}

\newcommand{\T}{\mathcal{T}}

\newcommand{\D}{\mathbb{D}}

\DeclareMathOperator{\ran}{ran}

\newcommand{\clb}{\mathcal{B}}

\newcommand{\cld}{\mathcal{D}}
\newcommand{\cle}{\mathcal{E}}

\newcommand{\clh}{\mathcal{H}}

\newcommand{\clk}{\mathcal{K}}

\newcommand{\clq}{\mathcal{Q}}
\newcommand{\clr}{\mathcal{R}}

\newcommand{\z}{\bm{z}}
\newcommand{\w}{\bm{w}}

\newcommand{\NI}{\noindent}

\begin{document}
	
	\title[{$q$-commuting dilation}]{A generalization of Ando's dilation, and isometric dilations for a class of tuples of $q$-commuting contractions}
	
\author[Barik]{Sibaprasad Barik}
\address{Department of Mathematics, Ben-Gurion University of the Negev, P.O.B. 653, Be'er Sheva 8410501, Israel}
\email{barik@post.bgu.ac.il, sibaprasadbarik00@gmail.com}
	
	
	\author[Bisai]{Bappa Bisai}
\address{Mathematics group, Harish-Chandra Research Institute, HBNI, Chhatnag Road, Jhunsi, Allahabad, 211019, India.} \email{bappabisai@hri.res.in, bappa.bisai1234@gmail.com}
	
	\begin{abstract}
Given a bounded operator $Q$ on a Hilbert space $\clh$, a pair of bounded operators $(T_1,T_2)$ on $\mathcal{H}$ is said to be $Q$-commuting if one of the following holds:
	 \[
	 T_1T_2=QT_2T_1  \text{ or }T_1T_2=T_2QT_1 \text{ or }T_1T_2=T_2T_1Q.	 \]
We give an explicit construction of isometric dilations for pairs of $Q$-commuting contractions for unitary $Q$, which generalizes the isometric dilation of Ando \cite{An} for pairs of commuting contractions. In particular, for $Q=qI_{\clh}$, where $q$ is a complex number of modulus $1$, this gives, as a corollary, an explicit construction of isometric dilations for pairs of $q$-commuting contractions which are well studied. There is an extended notion of $q$-commutativity for general tuples of operators and it is known that isometric dilation does not hold, in general, for an $n$-tuple of $q$-commuting contractions, where $n\geq 3$. Generalizing the class of commuting contractions considered by  Brehmer \cite{Bre}, we construct a class of $n$-tuples of $q$-commuting contractions and find isometric dilations explicitly for the class.
	\end{abstract}

\maketitle

\section*{Notations}

\begin{list}{\quad}{}
\item $\mathbb Z_{+}$\quad \quad \quad\ The set of all non negative integers.

\item $\mathbb Z_{+}^n$\quad \quad \quad\ The set of all $n$-tuple of non negative integers.

\item $\mathbb{D}$ \quad \quad \quad\ \ Open unit disc in the complex plane
$\mathbb{C}$ with center at origin.

\item  $\mathbb{D}^n$ \; \quad \quad\  \ Open unit polydisc in $\mathbb{C}^n$.

\item  $\clh$, $\clk$ \; \quad\ \ Hilbert spaces.

\item $\clb(\clh)$ \quad \;\ \ The space of all bounded linear
operators on $\clh$.




\end{list}

All Hilbert spaces are assumed to be over the complex numbers.

\newsection{Introduction}


\noindent One of the beautiful results in operator theory is the isometric dilation of a Hilbert space contraction (that is, an operator on a Hilbert space with norm not greater than 1) due to Sz.-Nagy \cite{Nagy}. Indeed, in 1953, Sz.-Nagy proved that for a contraction $T$ on a Hilbert space $\clh$, there exist a Hilbert space $\mathcal{K}$ $(\supset \mathcal{H})$ and an isometry $V$ on $\clk$ such that \[
	T^k=P_{\mathcal{H}}V^k|_{\mathcal{{H}}}
	\]
for all $k\in\Z$, where $P_{\mathcal{H}}$ denotes the orthogonal projection in $\clb(\clk)$ with range $\mathcal{H}$. A decade later, Ando \cite{An} generalized Sz.-Nagy's dilation result for pairs of commuting contractions, that is, he established that for a pair of commuting contractions $(T_1, T_2)$ on a Hilbert space $\mathcal{H}$, there exists a pair of commuting isometries $(V_1, V_2)$ on $\mathcal{K}$ ($\clk\supset\clh$) such that 
	\[
	T_1^{k_1}T_2^{k_2}=P_{\mathcal{H}}V_1^{k_1}V_2^{k_2}|_{\mathcal{{H}}} 
	\]
	for all $(k_1,k_2)\in\Z^2$. These dilation results of Sz.-Nagy and Ando are useful for analyzing the structures of bounded operators or pairs of bounded operators, respectively, on some Hilbert space. It is natural to ask whether we can extend Sz.-Nagy's dilation to arbitrary $n$-tuple of commuting contractions for $n\geq 3$. Producing some counterexample of a triple of commuting contractions, Parrott \cite{Par} showed that, for $n\geq3$, an $n$-tuple of commuting contractions does not have a dilation to an $n$-tuple of isometries, in general (also see \cite{CD}). There is growing literature on isometric dilation of different classes of $n$-tuples of commuting contractions considered by several authors. Some of the significant references are \cite{Bre}, \cite{MV}, \cite{CV1}, \cite{CV2}, \cite{Tim}, \cite{BDHS}, \cite{BDS}, \cite{BD} and \cite{NF}.
	
	Beyond the commutative setup, finding isometric dilations of non-commuting tuples of contractions is an active area of research. Isometric dilations of general non-commuting contractive tuples (that is, a non-commuting tuple $(T_1,\ldots,T_n)$ such that $\sum_{i=1}^nT_iT_i^*\leq I$, also known as row contractions) have been studied in \cite{P1} and \cite{P2}. The present paper deals with the isometric dilations of a special type of non-commuting tuples, known as the tuple of $q$-commuting operators, and defined as follows.
	\begin{defn}\label{qcommute}
		Suppose $q: \{1,\dots,n\}\times \{1,\dots,n\}\to \mathbb{T}$ is a function such that $q(i,i)=1$ and $q(i,j)=\overline{q(j,i)}$ for all $i,j=1,\dots, n$. An $n$-tuple of operators $(T_1,\dots,T_n)$ on a Hilbert space $\mathcal{H}$ is said to be $q$-\textit{commuting} if $T_iT_j=q(i,j)T_jT_i$ for all $i,j=1,\dots,n$. In particular, if $q(i,j)=-1$ for all $i,j=1,\dots,n$ and $i\neq j$, we call the tuple as \textit{anti-commuting} tuple.
	\end{defn}
	\NI These types of operator tuples have been seen to be useful in quantum theory and those have drawn the attention of many researchers in recent years. Trivially, every $n$-tuple of commuting operators are $q$-commuting, where $q(i,j)=1$ for all $i,j=1,\ldots,n$. For non-trivial examples of $q$-commuting operators, an interested reader is referred to \cite{KM}. For dilations and models of $q$-commuting row contractions see \cite{D} and \cite{BB}. For models of tuples of $q$-commuting isometries, one can see \cite{BS}.
	
	In \cite{S}, Sebesty\'en took a fascinating next step to extend the celebrated Ando's result for pairs of anti-commuting contractions. Indeed, by a commutant lifting approach, he proved that for a tuple of anti-commuting contractions $(T_1, T_2)$ on a Hilbert space $\mathcal{H}$, there is a Hilbert space $\mathcal{K}$ containing $\mathcal{H}$ and a pair of anti-commuting isometries $(V_1, V_2)$ on $\mathcal{K}$ such that 	\[
	T_1^{k_1}T_2^{k_2}=P_{\mathcal{H}}V_1^{k_1}V_2^{k_2}|_{\mathcal{{H}}} 
	\]
	for all $(k_1,k_2)\in\Z^2$.
	Recently, Keshari and Mallick \cite{KM} have generalized the result of Sebesty\'en by proving the existence of $q$-commuting isometric dilations for pairs of $q$-commuting contractions, and they have also done it through a commutant lifting approach. Mallick and Sumesh \cite{MS} improved Keshari and Mallick's $q$-commuting dilation result for pairs of $Q$-commuting contractions, where $Q$ is a unitary. Their definition of a pair of $Q$-commuting contractions is as follows.
	 
	 \begin{defn}\label{Q com}
	 Given $Q\in \mathcal{B}(\mathcal{H})$, a pair bounded operators $(T_1,T_2)$ on $\mathcal{H}$ is said to be $Q$-commuting if one of the following condition holds:
	 \[
	 T_1T_2=QT_2T_1 \text{ or }T_1T_2=T_2QT_1 \text{ or }T_1T_2=T_2T_1Q.	 \]
	 \end{defn}
\NI Examples of non trivial $Q$-commuting operators can be found in \cite{MS}. In \cite{MS}, the authors have shown that if $(T_1,T_2)$ is a pair contractions on a Hilbert space $\clh$ with the $Q$-commuting relation $T_1T_2=QT_2T_1$ (resp. $T_1T_2=T_2T_1Q$), then there exist a pair of isometries $(V_1,V_2)$ on some bigger Hilbert space $\clk$ with $V_1V_2=\widetilde{Q}V_2V_1$ (resp. $V_1V_2=V_2V_1\widetilde{Q}$) for some unitary $\widetilde{Q}$ which extends $Q$ and $(V_1,V_2)$ is a dilation of $(T_1, T_2)$ in the above sense. But, for the other case, that is, for $T_1T_2=T_2QT_1$, the existence of $\widetilde{Q}$-commuting isometric dilation is so far unknown. Also, there is no explicit construction of isometric dilations for a pair of $Q$-commuting contractions in the literature, even if for its well studied particular case of $q$-commuting pairs. In this article, we prove by an explicit construction that a pair of $Q$-commuting contractions can always be dilated to a pair of $\widetilde{Q}$-commuting isometries and further into a pair of $\widetilde{Q}$-commuting unitaries (see Theorem \ref{Qcommunidil}). Indeed, in Theorem \ref{Qcommdil}, for a pair of $Q$-commuting contractions $(T_1,T_2)$ on a Hilbert space $\mathcal{H}$, we explicitly construct a pair of $\widetilde{Q}$-commuting isometries $(V_1,V_2)$ on $\mathcal{K}$ $(\supset \mathcal{H})$ such that $(V_1,V_2)$ is a co-extension of $(T_1,T_2)$, that is, $V_i^*|_{\mathcal{H}}=T_i^*$ for $i=1,2$. From the $Q$-commutative case, as a corollary, we get an explicit construction of isometric dilations for $q$-commuting pairs. 

It is well known that, for $n\geq 3$, an $n$-tuple of $q$-commuting contractions does not possess an isometric dilation, in general, and this is due to the counterexample of Parrott \cite{Par} for commutative case. In this paper, we consider a class of $n$-tuples of $q$-commuting contractions which generalizes the class considered by Brehmer \cite{Bre} for the commutative setup and find their isometric dilations (see Theorem \ref{Bre-dil}). Moreover, we construct the dilation explicitly and see that the dilating tuples also lie in the same class. Here also, the dilations we find are co-extensions. In some particular case, our dilation result generalizes the isometric dilation result for a class of commuting tuples considered by Curto and Vasilescu \cite{CV1}.

Rest of the paper has two sections. Next section concerns with explicit construction of isometric dilations for pairs of $Q$-commuting contractions and consequently for $q$-commuting pairs. In the last section, we deal with general $n$-tuples ($n\geq 3$) of $q$-commuting contractions and construct isometric dilations explicitly for our class.
	
	\section{Dilation for pair of $Q$-commuting contractions}

\noindent This section is devoted to the construction of an explicit isometric dilation for a pair of a $Q$-commuting contractions, where $Q$ is a unitary operator. For this purpose we follow the line of construction of Ando \cite{An}.  Our main theorem of this section is as follows.

\begin{thm}\label{Qcommdil}
	Let $(T_1,T_2)$ be a  pair of $Q$-commuting contractions on a Hilbert space $\mathcal{{H}}$, where $Q\in\mathcal{B}(\mathcal{H})$ is a unitary. Then there exist a Hilbert space $\mathcal{K}$ containing $\mathcal{{H}}$, isometries $V_1,V_2 \in\mathcal{B}(\mathcal{K})$ and a unitary $\widetilde{Q}\in\mathcal{B}(\mathcal{K})$ such that $(V_1,V_2)$ is a pair of $\widetilde{Q}$-commuting isometries and 
	\[T_1^{k_1}T_2^{k_2}=P_{\mathcal{H}}V_1^{k_1}V_2^{k_2}|_{\mathcal{H}}\] 
	for all $k_1, k_2\in \Z$. Moreover, $\mathcal{H}$ is a reducing subspace for $\widetilde{Q}$ and $\widetilde{Q}|_{\mathcal{H}}=Q$.
\end{thm}
\begin{proof} We break the proof into three cases following the Definition \ref{Q com} of $Q$-commutativity.

\NI	\underline{Case-I:} \,First assume that $T_1T_2=QT_2T_1$. Let us consider the Hilbert space $\mathcal{K}=\bigoplus\limits_{0}^{\infty}\mathcal{H}$. We embed $\mathcal{H}$ in $\mathcal{K}$ via a map $h\mapsto (h,0,0,\dots)$. For a contraction $T$, let $D_T:=(I-T^*T)^{\frac{1}{2}}$. Define the operators $W_1$ and $W_2$ on $\mathcal{K}$ by 
	\begingroup
	\allowdisplaybreaks
	\begin{align*}
	W_1(h_0,h_1,h_2,h_3,\dots)&=(T_1h_0,D_{T_1}h_0,0,Qh_1,Q^2h_2,Q^3h_3,\dots)\\ 
	\text{and}\,\,W_2(h_0,h_1,h_2,h_3,\dots)&=(T_2h_0,D_{T_2}h_0,0,h_1,h_2,h_3,\dots).
	\end{align*}
	\endgroup
	The operators $W_1,W_2$ are isometries as $\|D_{T_i}h\|^2=\|h\|^2-\|T_ih\|^2$, for all $i=1,2$. Set $\widetilde{Q}=\bigoplus\limits_{0}^{\infty}Q$. Clearly, $\widetilde{Q}$ on $\mathcal{K}$ is a unitary. In general, $W_1W_2\neq \widetilde{Q}W_2W_1$. We reshape $W_1$ and $W_2$ in such a way that they give rise to an isometric dilation of $(T_1,T_2)$ and satisfy the $\widetilde{Q}$-commuting relation.
	
	Let us consider the space $\mathcal{G}=\mathcal{H}\oplus\mathcal{H}\oplus\mathcal{H}\oplus\mathcal{H}$ and identify $\mathcal{K}$ with $\mathcal{H}\oplus\big(\bigoplus\limits_0^{\infty}\mathcal{G}\big)$ by the natural identification
	\[
	(h_0,h_1,h_2,h_3,\dots)=(h_0,(h_1,h_2,h_3,h_4),(h_5,h_6,h_7,h_8),\dots).
	\]
	Define operators $R$ and $P$ on $\mathcal{G}$ by
	\begingroup
	\allowdisplaybreaks
	\begin{align*}
		R(h_i,h_j,h_k,h_m)&=(Qh_i,Qh_j,Qh_k,Qh_m)\\
	\text{and}\,\,	P(h_i,h_j,h_k,h_m)&=(h_i,Qh_j,Q^2h_k,Q^3h_m).
	\end{align*}
	\endgroup
	Needless to say, $R$ and $P$ are unitaries.
	Let $G:\mathcal{G}\to\mathcal{G}$ be a unitary, resolved later. Define $\widetilde{G_1}$ and $\widetilde{G_2}$ on $\mathcal{K}$ by
	\begingroup
	\allowdisplaybreaks
	\begin{align*}
	\widetilde{G_1}(h_0,h_1,h_2,\dots)&=(h_0,G(h_1,h_2,h_3,h_4),GR^{-3}(h_5,h_6,h_7,h_8),GR^{-6}(h_9,h_{10},h_{11},h_{12}),\dots),\\
	\widetilde{G_2}(h_0,h_1,h_2,\dots)&=(h_0,GP(h_1,h_2,h_3,h_4),GRP(h_5,h_6,h_7,h_8),GR^2P(h_9,h_{10},h_{11},h_{12}),\dots).
	\end{align*}
	\endgroup
	Clearly $\widetilde{G_1},\widetilde{G_2}$ are unitaries and 
	\begingroup
	\allowdisplaybreaks
	\begin{align*}
	&\widetilde{G_2}^{-1}(h_0,h_1,h_2,h_3,\dots)\\
	=&\left(h_0,(GP)^{-1}(h_1,h_2,h_3,h_4),(GRP)^{-1}(h_5,h_6,h_7,h_8),(GR^2P)^{-1}(h_9,h_{10},h_{11},h_{12}),\dots\right).
	\end{align*}
	\endgroup
	Set $V_1=\widetilde{G_1}W_1$ and $V_2=W_2\widetilde{G_2}^{-1}$. Clearly $V_1$ and $V_2$ are isometries. Our aim is to find the unitary $G$ on $\mathcal{G}$ such that $V_1V_2=\widetilde{Q}V_2V_1$. To this end, first we write down $V_1$ and $V_2$ separately.
	\begin{align*}
	V_1&(h_0,h_1,h_2,h_3,\dots)\\
	&=\widetilde{G_1}W_1(h_0,h_1,h_2,h_3,\dots)\\
	&=\widetilde{G_1}(T_1h_0,D_{T_1}h_0,0,Qh_1,Q^2h_2,Q^3h_3, \dots)\\
	&=\big(T_1h_0,G(D_{T_1}h_0,0,Qh_1,Q^2h_2),GR^{-3}R^3P(h_3,h_4,h_5,h_6),\\
		&\hspace{5cm} GR^{-6}R^7P(h_7,h_8,h_9,h_{10}), GR^{-9}R^{11}P(h_{11},h_{12},h_{13},h_{14}),\dots\big)\\
&=\big(T_1h_0,G(D_{T_1}h_0,0,Qh_1,Q^2h_2),GP(h_3,h_4,h_5,h_6),\\
		&\hspace{5cm} GRP(h_7,h_8,h_9,h_{10}), GR^2P(h_{11},h_{12},h_{13},h_{14}),\dots\big).
	\end{align*}
	Also,
	\begin{align*}
	V_2&(h_0,h_1,h_2,h_3,\dots)\\
	&=W_2\widetilde{G_2}^{-1}(h_0,h_1,h_2,h_3,\dots)\\
	&=W_2\big(h_0,(GP)^{-1}(h_1,h_2,h_3,h_4),(GRP)^{-1}(h_5,h_6,h_7,h_8),(GR^2P)^{-1}(h_9,h_{10},h_{11},h_{12}),\dots\big)\\
	&=\big(T_2h_0,D_{T_2}h_0,0,(GP)^{-1}(h_1,h_2,h_3,h_4),\\
	&\hspace{5cm} (GRP)^{-1}(h_5,h_6,h_7,h_8), (GR^2P)^{-1}(h_9,h_{10},h_{11},h_{12}),\dots\big).
	\end{align*}
	Using these expressions of $V_1$ and $V_2$, we calculate $V_1V_2$ and $\widetilde{Q}V_2V_1$.
	\begin{align*}
	V_1V_2&(h_0,h_1,h_2,h_3,\dots)\\
	=&V_1\big(T_2h_0,D_{T_2}h_0,0,(GP)^{-1}(h_1,h_2,h_3,h_4),(GRP)^{-1}(h_5,h_6,h_7,h_8),\\
	& \qquad\qquad\qquad\qquad\qquad\qquad\qquad\qquad\qquad\qquad\qquad\qquad(GR^2P)^{-1}(h_9,h_{10},h_{11},h_{12}),\dots\big)\\
	=&\big(T_1T_2h_0,G(D_{T_1}T_2h_0,0,QD_{T_2}h_0,0),(h_1,h_2,h_3,h_4),(h_5,h_6,h_7,h_8),(h_9,h_{10},h_{11},h_{12}),\dots\big).
	\end{align*}
On the other hand,
\begin{align*}
	\widetilde{Q}V_2V_1&(h_0,h_1,h_2,h_3,\dots)\\
	=&\widetilde{Q}W_2\widetilde{G_2}^{-1}\big(T_1h_0,G(D_{T_1}h_0,0,Qh_1,Q^2h_2),GP(h_3,h_4,h_5,h_6),GRP(h_7,h_8,h_9,h_{10}),\\
	&\qquad\qquad\qquad\qquad\qquad\qquad\qquad\qquad GR^2P(h_{11},h_{12},h_{13},h_{14}),\dots\big)\\
	=&\widetilde{Q}W_2\big(T_1h_0,(GP)^{-1}G(D_{T_1}h_0,0,Qh_1,Q^2h_2),(GRP)^{-1}(GP)(h_3,h_4,h_5,h_6),\\
	&\qquad\qquad\quad(GR^2P)^{-1}(GRP)(h_7,h_8,h_9,h_{10}),
	(GR^3P)^{-1}(GR^2P)(h_{11},h_{12},h_{13},h_{14}),\dots\big)\\
   =&\widetilde{Q}W_2\big(T_1h_0,(D_{T_1}h_0,0,Q^{-1}h_1,Q^{-1}h_2),R^{-1}(h_3,h_4,h_5,h_6),R^{-1}(h_7,h_8,h_9,h_{10}),\dots\big) \\
   &\hspace{11cm} [\text{since }RP=PR]\\
	=&\widetilde{Q}\big(T_2T_1h_0,(D_{T_2}T_1h_0,0,D_{T_1}h_0,0),Q^{-1}(h_1,h_2,h_3,h_4),Q^{-1}(h_5,h_6,h_7,h_8),\dots\big)\\
	=&\big(QT_2T_1h_0,(QD_{T_2}T_1h_0,0,QD_{T_1}h_0,0),(h_1,h_2,h_3,h_4),(h_5,h_6,h_7,h_8),\dots\big).
	\end{align*}

	\NI Since $T_1T_2=QT_2T_1$, therefore, $V_1V_2=\widetilde{Q}V_2V_1$ if and only if 
	\begin{equation}\label{eqn1}
	G(D_{T_1}T_2h_0,0,QD_{T_2}h_0,0)=(QD_{T_2}T_1h_0,0,QD_{T_1}h_0,0) \text{ for all }h_0\in\mathcal{H}.
	\end{equation}
	
	\noindent Suppose 
	\[\mathcal{L}_1=\overline{\left\{(D_{T_1}T_2h_0,0,QD_{T_2}h_0,0):h_0\in \mathcal{H}\right\}}\quad \text{and}\,\, \mathcal{L}_2=\overline{\left\{(QD_{T_2}T_1h_0,0,QD_{T_1}h_0,0):h_0\in \mathcal{H}\right\}}.\]
	Since
	\begingroup
	\allowdisplaybreaks
	\begin{align*}
	 T_2^*D^2_{T_1}T_2+D_{T_2}Q^*Q D_{T_2} =&T_2^*D^2_{T_1}T_2+D^2_{T_2}\\ =&T_2^*(I-T_1^*T_1)T_2+(I-T_2^*T_2)\\
	 =&T_1^*(I-T_2^*Q^*QT_2)T_1+(I-T_1^*T_1)\quad\quad[\text{using}\,\, T_1T_2=QT_2T_1]\\
	 =&T_1^*(I-T_2^*T_2)T_1+(I-T_1^*T_1)\\
	 =&T_1^*D^2_{T_2}T_1+D^2_{T_1}\\
	 =&T_1^*D_{T_2}Q^*QD_{T_2}T_1+D_{T_1}Q^*QD_{T_1},
	 \end{align*} 
	 \endgroup
	 we have
	\[\|D_{T_1}T_2h_0\|^2+\|QD_{T_2}h_0\|^2= \|QD_{T_2}T_1h_0\|^2+\|QD_{T_1}h_0\|^2\]
	for all $h_0\in \mathcal{H}$.
	So,  the operator 
	$
	G:\mathcal{L}_1\to \mathcal{L}_2
	$
	defined by
	\[
	G(D_{T_1}T_2h_0,0,QD_{T_2}h_0,0)=(QD_{T_2}T_1h_0,0,QD_{T_1}h_0,0)
	\]
	is an isometry. It is easy to check that $\text{dim }\mathcal{G}\ominus \mathcal{L}_1=\text{dim }\mathcal{G}\ominus \mathcal{L}_2$. Therefore, $G$ extends to a unitary, denoted again by $G$, on $\mathcal{G}$ such that \eqref{eqn1} holds. Thus $V_1V_2=\widetilde{Q}V_2V_1$. It is clear from the above expression of $V_1$ and $V_2$ that
	\[
	V_i^{k_i}(h_0,h_1,h_2,\dots)=(T_i^{k_i}h_0,\dots) 
	\]
	for all $i=1,2$ and $k_i\in\Z$. Hence
	\[T_1^{k_1}T_2^{k_2}=P_{\mathcal{H}}V_1^{k_1}V_2^{k_2}|_{\mathcal{H}}\]
	for all $k_1,k_2\in\Z$.
	
	\noindent \underline{Case-II:} \,Suppose $T_1T_2=T_2QT_1$. For this case, consider \[\mathcal{L}_1=\overline{\left\{(D_{T_1}T_2h_0,0,QD_{T_2}h_0,0):h_0\in \mathcal{H}\right\}}\quad \text{and}\,\, \mathcal{L}_2=\overline{\left\{(D_{T_2}QT_1h_0,0,QD_{T_1}h_0,0):h_0\in \mathcal{H}\right\}}.\]
	Since 
	\begingroup
	\allowdisplaybreaks
	\begin{align*}
	 T_2^*D^2_{T_1}T_2+D_{T_2}Q^*Q D_{T_2} =&T_2^*(I-T_1^*T_1)T_2+(I-T_2^*T_2)\\
	 =& I-T_2^*T_1^*T_1T_2\\
	 =&T_1^*Q^*(I-T_2^*T_2)QT_1+(I-T_1^*T_1)\\
	 =&T_1^*Q^*D^2_{T_2}QT_1+D_{T_1}Q^*QD_{T_1},
	 \end{align*} 
	 \endgroup
	 we have 
	\[\|D_{T_1}T_2h_0\|^2+\|QD_{T_2}h_0\|^2=\|D_{T_2}QT_1h_0\|^2+\|QD_{T_1}h_0\|^2\]
	for all $h_0\in \mathcal{H}$. So, the operator 
	$
	G:\mathcal{L}_1\to \mathcal{L}_2
	$
	defined by
	\[
	G(D_{T_1}T_2h_0,0,QD_{T_2}h_0,0)=(D_{T_2}QT_1h_0,0,QD_{T_1}h_0,0)
	\]
	is an isometry. Therefore, $G$ extends to a unitary, denoted again by $G$, on $\mathcal{G}$ such that $G(D_{T_1}T_2h_0,0,QD_{T_2}h_0,0)=(D_{T_2}QT_1h_0,0,QD_{T_1}h_0,0)$ for all $h\in \mathcal{H}$, where $\mathcal{G}$ is defined as in Case-I. Also let $W_1,W_2,R,P,\widetilde{G}_1,\widetilde{G}_2, V_1$ and $V_2$ be all as in Case-I. So, $V_1V_2$ is as in the previous case. Set $\widetilde{Q}= Q \oplus\bigoplus\limits_{0}^{\infty}GRG^{-1}$ on $\mathcal{K}=\mathcal{{H}}\oplus\bigoplus\limits_{0}^{\infty}\mathcal{G}$. We calculate
	\begingroup
	\allowdisplaybreaks
	\begin{align*}
	V_2\widetilde{Q}V_1&(h_0,h_1,h_2,h_3,\dots)\\
	=&W_2\widetilde{G_2}^{-1}\widetilde{Q}\big(T_1h_0,G(D_{T_1}h_0,0,Qh_1,Q^2h_2),GP(h_3,h_4,h_5,h_6),GRP(h_7,h_8,h_9,h_{10}),\\
	&\qquad\qquad\qquad\qquad\qquad\qquad\qquad\qquad\qquad GR^2P(h_{11},h_{12},h_{13},h_{14}),\dots\big)\\
	=&W_2\widetilde{G_2}^{-1}\big(QT_1h_0,GR(D_{T_1}h_0,0,Qh_1,Q^2h_2),GRP(h_3,h_4,h_5,h_6),GR^2P(h_7,h_8,h_9,h_{10}),\\
	&\qquad\qquad\qquad\qquad\qquad\qquad\qquad\qquad\qquad GR^3P(h_{11},h_{12},h_{13},h_{14}),\dots\big)\\
   =&W_2\big(QT_1h_0,P^{-1}R(D_{T_1}h_0,0,Qh_1,Q^2h_2),(h_3,h_4,h_5,h_6),(h_7,h_8,h_9,h_{10}),\dots\big) \\
    =&W_2\big(QT_1h_0,(QD_{T_1}h_0,0,h_1,h_2), (h_3,h_4,h_5,h_6), (h_7,h_8,h_9,h_{10}),\dots\big) \\
	=&\big(T_2QT_1h_0,(D_{T_2}QT_1h_0,0,QD_{T_1}h_0,0),(h_1,h_2,h_3,h_4),(h_5,h_6,h_7,h_8),\dots\big)\\
	=&\big(T_1T_2h_0,G(D_{T_1}T_2h_0,0,QD_{T_2}h_0,0),(h_1,h_2,h_3,h_4),(h_5,h_6,h_7,h_8),\dots\big)
\end{align*}
	\endgroup
This shows that $V_1V_2=V_2\widetilde{Q}V_1$. Also, from the previous case it follows that $(V_1,V_2)$ is an isometric dilation of $(T_1,T_2)$.

\noindent\underline{Case-III:} \,Suppose, $T_1T_2=T_2T_1Q$. In this case, we take $V_1,V_2$ and $\widetilde{Q}$ same as in Case-I and get the required result.
\end{proof}

As a corollary of the above theorem we get the following dilation result for pairs of $q$-commuting contractions where $|q|=1$.

\begin{cor}\label{qcommdil}
	Let $(T_1,T_2)$ be a pair of contractions on a Hilbert space such that $T_1T_2=qT_2T_1$, where $q$ is a unimodular complex number. Then there exist a Hilbert space $\mathcal{K}$ containing $\mathcal{{H}}$ and a pair of isometries $(V_1, V_2)$ on $\mathcal{K}$ such that $V_1V_2=qV_2V_1$ and
	\[ 
	T_1^{k_1}T_2^{k_2}=P_{\mathcal{H}}V_1^{k_1}V_2^{k_2}|_{\mathcal{H}}\]
for all $k_1,k_2\in\Z$.
\end{cor}
\begin{proof}
    If we consider $Q=qI_{\mathcal{H}}$ and $\widetilde{Q}=qI_{\mathcal{K}}$, then it directly follows from Theorem \ref{Qcommdil}.
\end{proof}

\begin{remark}
In the particular case when $q=1$, we recover the Ando's dilation theorem where the explicit construction matches with the original construction of Ando.
\end{remark}

The following proposition shows that every pair of $Q$-commuting isometries admits a $\widetilde{Q}$-commuting unitary extension. We write the proposition in two separate parts. Part (i) has been proved in \cite{MS} (see Theorem 2.18 and Corollary 2.19 in \cite{MS}) and we include it in the statement for the sake of completeness. We only prove part (ii) of the proposition and for this we follow the same technique as in the proof of Theorem 2.18 in \cite{MS}.

\begin{propn}
    Let $(V_1,V_2)$ be a pair of isometries on a Hilbert space $\clh$ and $Q\in\mathcal{B}(\mathcal{H})$ be a unitary.
    \begin{itemize}
        \item[(i)] If $V_1V_2=QV_2V_1$ (resp. $V_1V_2=V_2V_1Q$), then there exist a Hilbert space $\mathcal{K}$ $(\supset \mathcal{H})$, a unitary $\widetilde{Q}$ with $\widetilde{Q}|_{\clh}=Q$ and a pair of unitaries $(U_1,U_2)$ on $\mathcal{K}$ such that $U_1U_2=\widetilde{Q}U_2U_1$ (resp. $U_1U_2=U_2U_1\widetilde{Q}$) and $U_i|_{\mathcal{H}}=V_i$ for all $i=1,2$.
        \item[(ii)] If $V_1V_2=V_2QV_1$, then there exist a Hilbert space $\mathcal{K}$ $\supset\mathcal{H}$, a unitary $\widetilde{Q}$ with $\widetilde{Q}|_{\clh}=Q$ and a pair of unitaries $(U_1,U_2)$ on $\mathcal{K}$ such that $U_1U_2=U_2\widetilde{Q}U_1$ and $U_i|_{\mathcal{H}}=V_i$ for all $i=1,2$.
    \end{itemize}
    In fact, we can choose $\widetilde{Q}=Q\oplus I|_{\mathcal{K}\ominus \mathcal{H}}$ in all cases.
\end{propn}
\begin{proof}
     We prove part (ii) only. Let $V_2$ extends to some unitary $\widetilde{V_2}$ on $\widetilde{\mathcal{H}}$ such that the extension is minimal, that is,  $\widetilde{\clh}=\overline{\text{span}}\{\widetilde{V_2}^n\mathcal{H}:n\in \mathbb{Z}\}$. Set $\widehat{Q}=(Q\oplus I_{\widetilde{\clh}\ominus \mathcal{H}})\in \mathcal{B}(\widetilde{\mathcal{H}})$. Define $\widetilde{V_1}: \widetilde{\mathcal{H}}\to \widetilde{\mathcal{H}}$ by 
    \[
    \widetilde{V_1}(\widetilde{V_2}^nh)=(\widetilde{V_2}\widehat{Q})^nV_1h \quad (h\in \mathcal{H}, n\in \mathbb Z). 
    \]
 Now, for all $h,h'\in \mathcal{H}$ and $n\geq m$,
    \begingroup
    \allowdisplaybreaks
    \begin{align*}
        \langle \widetilde{V_1}(\widetilde{V_2}^nh), \widetilde{V_1}(\widetilde{V_2}^mh') \rangle&=\langle (\widetilde{V_2}\widehat{Q})^nV_1h, (\widetilde{V_2}\widehat{Q})^mV_1h' \rangle\\
        &=\langle (\widetilde{V_2}\widehat{Q})^{n-m}V_1h, V_1h' \rangle\\
        &=\langle (V_2Q)^{n-m}V_1h,V_1h' \rangle\hspace{2cm} \big[\text{since}\, (V_2Q)^n= (\widetilde{V_2}\widehat{Q})^n|_{\mathcal{H}}\big]\\
        &=\langle V_1V_2^{n-m}h, V_1h' \rangle \hspace{3.6cm} [\text{using }V_1V_2=V_2QV_1]\\
        &=\langle V_2^{n-m}h,h' \rangle \hspace{3.9cm} [\text{since $V_1$ is an isometry}]\\
        &=\langle \widetilde{V_2}^{n-m}h,h' \rangle \hspace{5.2cm}[\text{since } \widetilde{V_2}|_{\mathcal{H}}=V_2]\\
        &=\langle \widetilde{V_2}^nh, \widetilde{V_2}^mh' \rangle.
    \end{align*}
    \endgroup
    Therefore, $\widetilde{V_1}$ on $\widetilde{\mathcal{H}}$ is an isometric extension of $V_1$ on $\mathcal{H}$. It is clear that $\widetilde{V_1}\widetilde{V_2}=\widetilde{V_2}\widehat{Q}\widetilde{V_1}$. Suppose, $U_1$ on $\mathcal{K}$ is the minimal unitary extension of $\widetilde{V_1}$ on $\widetilde{\mathcal{H}}$. Then,
    \[\mathcal{K}=\overline{\text{span}}\big\{U_1^n\widetilde{\mathcal{H}}: n\in \mathbb Z\big\}=\overline{\text{span}}\big\{U_1^{*n}\widetilde{\mathcal{H}}: n\in\Z\big\},\]
    as $\widetilde{H}$ is an invariant subspace for $U_1$. Define $U_2:\mathcal{K}\to \mathcal{K}$, through its adjoint, by
    \[
    U_2^*(U_1^{*n}k)=(\widetilde{Q}U_1)^{*n}\widetilde{V_2}^*k \quad ( k\in \widetilde{\mathcal{H}}, n\in\Z),
    \]
    where $\widetilde{Q}=(Q\oplus I_{\mathcal{K}\ominus\mathcal{H}})\in \mathcal{B}(\mathcal{K})$. Then, for $k,k'\in \widetilde{\clh}$ and $n\geq m$,
    \begingroup
    \allowdisplaybreaks
    \begin{align*}
        \langle U_2^*(U_1^{*n}k), U_2^*(U_1^{*m}k') \rangle&= \langle (\widetilde{Q}U_1)^{*n}\widetilde{V_2}^*k,(\widetilde{Q}U_1)^{*m}\widetilde{V_2}^*k' \rangle\\
        &=\langle (\widetilde{Q}U_1)^{*(n-m)}\widetilde{V_2}^*k, \widetilde{V_2}^*k' \rangle\\
        &=\langle (\widehat{Q}\widetilde{V_1})^{*(n-m)}\widetilde{V_2}^*k,\widetilde{V_2}^*k' \rangle \quad \big[\text{since }(\widehat{Q}\widetilde{V_1})^{*l}=P_{\widetilde{\mathcal{H}}}(\widetilde{Q}U_1)^{*l}|_{\widetilde{\mathcal{H}}}\big]\\
        &=\langle \widetilde{V_2}^*\widetilde{V_1}^{*(n-m)}k, \widetilde{V_2}^*k' \rangle \hspace{2.6cm} \big[\text{using }\widetilde{V_1}\widetilde{V_2}=\widetilde{V_2}\widehat{Q}\widetilde{V_1}\big]\\
        &=\langle \widetilde{V_1}^{*(n-m)}k, k' \rangle \\
        &=\langle U_1^{*(n-m)}k,k' \rangle \hspace{3.9cm} \big[\text{since }\widetilde{V_1}^*=P_{\widetilde{\mathcal{K}}}U_1^*|_{\widetilde{\mathcal{K}}}\big]\\
        &=\langle U_1^{*n}k,U_1^{*m}k' \rangle.
    \end{align*}
    \endgroup
    Therefore, $U_2^*$ on $\mathcal{K}$ is an isometric extension of $\widetilde{V_2}^*$ on $\widetilde{\mathcal{H}}$. Now we shall prove that $U_2^*$ on $\mathcal{K}$ is unitary. Since $U_2^*$ is isometry, it suffices to show that $U_2^*$ is onto. For $n>0$, set
    \[
    \mathcal{K}_n=\overline{\text{span}}\left\{U_1^{*j}\widetilde{\mathcal{H}}: 0\leq j
    \leq n\right\}.
    \]
    We prove by induction that $U_2^*(\mathcal{K}_n)=\mathcal{K}_n$, for all $n>0$. Now for $h\in \widetilde{\mathcal{H}}$, $U_1^*\widetilde{V_2}\widetilde{Q}h \in \mathcal{K}_1$ and 
    \[
    U_2^*\left(U_1^*\widetilde{V_2}\widetilde{Q}h\right)=\left(\widetilde{Q}U_1\right)^*\widetilde{V_2}^*\widetilde{V_2}\widetilde{Q}h=U_1^*\widetilde{Q}^*\widetilde{Q}h=U_1^*h.
    \]
    This implies that $U_2^*\left(\mathcal{K}_1\right)=\mathcal{K}_1$. Now assume that $U_2^*(\mathcal{K}_n)=\mathcal{K}_n$. To prove $U_2^*(\mathcal{K}_{n+1})=\mathcal{K}_{n+1}$, it is enough to show that for $h\in \widetilde{\mathcal{H}}$, $U_1^{*(n+1)}h$ has a pre-image under the operator $U_2^*$. As $U_1^{*n}h\in \mathcal{K}_n$, there exists $y\in \mathcal{K}_n$ such that $U_2^*(y)=U_1^{*n}h$. Again, since $\mathcal{K}_n$ is a reducing subspace for $\widetilde{Q}$, there exists $z\in \mathcal{K}_n$ such that $U_2^*(z)=\widetilde{Q}U_2^*(y)$. Now it is clear that $U_1^*(z)\in \mathcal{K}_{n+1}$ and 
    \[
    U_2^*\left(U_1^*z\right)=U_1^*\widetilde{Q}^*U_2^*(z)=U_1^*U_2^*(y)=U_1^{*(n+1)}(h).
    \]
    Therefore, $U_2^*(\mathcal{K}_{n+1})=\mathcal{K}_{n+1}$ and hence $U_2$ is unitary. Using the facts that $U_2^*|_{\widetilde{\mathcal{H}}}=\widetilde{V_2}^*$ and $\widetilde{V_2}^*$ on $\widetilde{\mathcal{H}}$ is a unitary one can easily prove that the block matrix form of $U_2^*$, with respect to the decomposition $\mathcal{K}=\widetilde{H}\oplus(\mathcal{K}\ominus\widetilde{\mathcal{H}})$, will be a diagonal block matrix. Thus $U_2$ is an extension of $\widetilde{V_2}$. This completes the proof.
\end{proof}

Combining Theorem \ref{Qcommdil} and Proposition \ref{Qcommunidil}, we have the following $\widetilde{Q}$-commuting unitary dilation for a pair of $Q$-commuting contractions.
\begin{thm}\label{Qcommunidil}
    Let $(T_1,T_2)$ be a pair of $Q$-commuting contractions on a Hilbert space $\mathcal{{H}}$, where $Q\in\mathcal{B}(\mathcal{H})$ is a unitary. Then, there exist a Hilbert space $\mathcal{K}\supset\mathcal{{H}}$, unitaries $U_1,U_2$ and $\widetilde{Q} \in\mathcal{B}(\mathcal{K})$  such that $(U_1,U_2)$ is a pair of $\widetilde{Q}$-commuting unitaries and 
	\[T_1^{k_1}T_2^{k_2}=P_{\mathcal{H}}U_1^{k_1}U_2^{k_2}|_{\mathcal{H}}\] 
	for all $k_1, k_2\in \Z$. Moreover, $\mathcal{H}$ is a reducing subspace for $\widetilde{Q}$ and $\widetilde{Q}|_{\mathcal{H}}=Q$.
\end{thm}

\section{Dilation for a class of $q$-commuting tuples}\label{tuple case}

\NI In this section, we shall work with the tuples of $q$-commuting contractions. For a given $q: \{1,\dots,n\}\times \{1,\dots,n\}\to \mathbb{T}$ as is Definition \ref{qcommute}, first we denote $\T^{q,n}(\clh)$ as the set of all $n$-tuple of $q$-commuting contractions on a Hilbert space $\clh$, that is,
\[
\T^{q,n}(\clh):=\{T=(T_1,\ldots,T_n)\in \mathcal{B}(\mathcal{H})^n: \|T_i\|\leq 1\text{ and }T_iT_j=q(i,j)T_jT_i\,\,\,\forall\,i,j=1\ldots,n\}.
\]
We shall find isometric dilation for a class of tuples lies in $\T^{q,n}(\clh)$, which generalizes the dilation result for the class of Curto and Vasilescu \cite{CV1} and Brehmer \cite{Bre}, obtained for the commuting case. We define our class later in this section. In Theorem  \ref{Bre-dil} below, we describe dilation for that class. Very soon we shall define some operators, namely $q$-commuting rotational shifts, on the Hardy space over the unit polydisc, which are crucial in describing dilations for our class. We start with recalling the well-known Hardy space and shift operators on it.

\textbf{Hardy space and shift operators:} The Hardy space over the unit polydisc $\D^n$, denoted by $H^2(\D^n)$, is the reproducing kernel Hilbert space with respect to the kernel function $\mathbb{S}_n(\z,\w)$, where $\mathbb{S}_n:\D^n\times\D^n\to \mathbb C$ is the Szeg\"o kernel over the unit polydisc, defined by
\[
\mathbb{S}_n(\z,\w)=\prod_{i=1}^n\frac{1}{1-z_i\bar{w_i}}\quad (\z=(z_1,\ldots,z_n),\,\w=(w_1,\dots,w_n)\in \D^n).
\]
In the vector valued case, for a Hilbert space $\cle$, the $\cle$-valued Hardy space over the unit polydisc $\D^n$, denoted by $H^2_{\cle}(\D^n)$, is the reproducing kernel Hilbert space with respect to the kernel function $\mathbb{S}_n(\z,\w)\cdot I_{\cle}$. One can also concretely describe the space $H^2_{\cle}(\D^n)$ as follows:
\[
H^2_{\cle}(\D^n)=\Big\{\sum_{\bm k\in\Z^n} a_{\bm k}\z^{\bm k}: a_{\bm k}\in\cle\,\,\text{and}\, \sum_{\bm k\in\Z^n}\|a_{\bm k}\|^2<\infty\Big\}.
\]
Note that, the space $H^2_{\cle}(\D^n)$ can be viewed as $H^2(\D^n)\otimes \cle$. Shift operators are well-known isometries on the Hardy space $H^2_{\cle}(\D^n)$. With respect to each co-ordinate function $z_i$, $i=1,\dots,n$, the $i$-th shift operator $M_{z_i}: H^2_{\cle}(\D^n)\to H^2_{\cle}(\D^n)$ is defined by
\[
M_{z_i}f(\w)=w_if(\w)\quad (\w\in\D^n, f\in H^2_{\cle}(\D^n)).
\]
It is easy to check that, $(M_{z_1},\ldots,M_{z_n})$ is an $n$-tuple of doubly commuting isometries.

\textbf{Szeg\"o positivity:} Looking at the Szeg\"o kernel over $\D^n$, one can check inverse of it is a polynomial, more precisely
\begin{equation}\label{szegopoly}
\mathbb{S}_n^{-1}(\z,\w)=(1-z_1\bar{w}_1)\cdots(1-z_n\bar{w}_n)=\sum_{\bm k\in\Z^n,\, 0\leq \bm k\leq \bm e}(-1)^{|\bm k|}\z^{\bm k}\bar{\w}^{\bm k},
\end{equation}
where, $\bm e=(1,\ldots,1)\in\Z^n$ and for $\bm k=(k_1,\ldots,k_n)\in \mathbb{Z}^n_+$, $\z^{\bm k}=z_1^{k_1}\cdots z_n^{k_n}$ and $\bar{\w}^{\bm k}=\bar{w}_n^{k_n}\cdots\bar{w}_1^{k_1}$. 
For $T\in\T^{q,n}(\clh)$, connecting with \eqref{szegopoly}, we define
\[
\mathbb{S}_n^{-1}(T,T^*):=	\sum_{F\subset\{1,\ldots,n\}}(-1)^{|F|}T_FT_F^*
\]
 where, for $T=(T_1,\ldots,T_n)$ and for an ordered subset $F=\{k_1,\ldots,k_r\}\subset\{1,\ldots,n\}$, $T_F$ and $T_F^*$ are the notations for
 \[
 T_F:=T_{k_1}\cdots T_{k_r} \text{ and } T_F^*:=T_{k_r}^*\cdots T_{k_1}^*.
 \]
 Note that, in the notations $T_F$ and $T_F^*$, maintaining the order of appearance of the operators, involve in the composition, is important and we fix these notations for the rest of the article. We say that $T\in\T^{q,n}(\clh)$ satisfies \textit{Szeg\"o positivity} if 
 \[
\mathbb{S}_n^{-1}(T,T^*)\geq 0.
	\]
It follows from the doubly commutativity of the tuple $(M_{z_1},\ldots, M_{z_n})$ on $H^2_{\cle}(\D^n)$ that it satisfies the Szeg\"o positivity. Next proposition gives us an example of a large class of operators tuples which satisfies Szeg\"o positivity. Before going into the proposition we make a definition.
\begin{defn}
Let $q$ be as in Definition \ref{qcommute}. An $n$-tuple of operators $(T_1,\dots,T_n)$ on a Hilbert space $\mathcal{H}$ is said to be \textit{doubly $q$-\textit{commuting}} if $T_iT_j=q(i,j)T_jT_i$ and $T_iT_j^*=\overline{q(i,j)}T_j^*T_i$ for all $i,j=1,\dots,n$. 
\end{defn}

\NI With the above definition we have the following proposition.
\begin{propn}\label{doublyszego}
Let $T\in\T^{q,n}(\clh)$ be doubly $q$-commuting. Then, $T$ satisfies Szeg\"o positivity.
\end{propn}
\begin{proof} Let $\emptyset\neq F=\{n_1,\ldots,n_k\}\subset\{1,\ldots,n\}$ be an ordered subset. Then,
\begingroup
\allowdisplaybreaks
\begin{align*}
    T_FT_F^*=&  T_{n_1}T_{F\setminus \{n_1\}}T_{F\setminus \{n_1\}}^*T_{n_1}^*\\
    =& \Big(\prod_{m\in F\setminus\{n_1\}} \overline{q(n_1,m)}\Big)T_{n_1}T_{F\setminus \{n_1\}}T_{n_1}^*T_{F\setminus \{n_1\}}^*\\
    =&\Big( \prod_{m\in F\setminus\{n_1\}} \overline{q(n_1,m)}\overline{q(m,n_1)}\Big)T_{n_1}T_{n_1}^*T_{F\setminus \{n_1\}}T_{F\setminus \{n_1\}}^*\\
    =& T_{n_1}T_{n_1}^*T_{F\setminus \{n_1\}}T_{F\setminus \{n_1\}}^*\\
    =& \prod_{p\in F} T_pT_p^*.\hspace{1.5cm}[\text{Repeating the previous steps for remaining elements in $F$}]
\end{align*}
\endgroup
Now,
\begingroup
\allowdisplaybreaks
    \begin{align*}
\mathbb{S}_n^{-1}(T,T^*)=&	\sum_{F\subset\{1,\ldots,n\}}(-1)^{|F|}T_FT_F^*
= \sum_{F\subset\{1,\ldots,n\}}(-1)^{|F|}\prod_{p\in F} T_pT_p^*
=\prod_{p\in F}(I- T_pT_p^*)\geq 0.
    \end{align*}
    \endgroup
The last inequality follows as $T_p$'s are contractions. Hence proved.
\end{proof}

\textbf{$q$-commuting rotational shifts:} We define an $n$-tuple of $q$-commuting operators on the Hardy space $H^2_{\cle}(\D^n)$ as follows: For a given function $q$ as in Definition \ref{qcommute} we denote $[q]_m$, as a tuple in $\mathbb{C}^n$, by 
\[[q]_m=\big(q(m,1)^{\frac{1}{2}},\ldots,q(m,n)^{\frac{1}{2}}\big)\]
for all $m=1,\ldots,n$. With this tuple we define a rotation operator $R_{[q]_m}: H^2_{\cle}(\D^n)\to H^2_{\cle}(\D^n)$ by 
\begin{equation*}
    R_{[q]_m}f(\bm z)=f([q]_m\bm z),\quad \quad \big(f\in H^2_{\cle}(\D^n)\big)
\end{equation*}
where for $\bm z=(z_1,\ldots,z_n)$, $[q]_m\bm z=\big(q(m,1)^{\frac{1}{2}}z_1,\ldots,q(m,n)^{\frac{1}{2}}z_n\big)$. More precisely,
\begin{equation}\label{rotation}
    R_{[q]_m}\Big(\sum_{\bm{k}\in\Z^n}a_{\bm{k}}\z^{\bm{k}}\Big)=\sum_{\bm{k}\in\Z^n}a_{\bm{k}}[q]_m^{\bm{k}}\z^{\bm{k}}
\end{equation}
where for $\bm k=(k_1,\ldots,k_n)$, $[q]_m^{\bm{k}}=\prod_{i=1}^n q(m,i)^{\frac{k_i
}{2}} $.
Now, consider the $n$-tuple of operators
\[
(M_{z_1} R_{[q]_1},\ldots, M_{z_n}R_{[q]_n})  \text{ on }  H^2_{\cle}(\D^n)
\]
for some Hilbert space $\cle$, where $M_{z_i}$'s, for $i=1,\ldots,n$ are the shift operators on the Hardy space $H^2_{\cle}(\D^n)$ as defined earlier. We call this kind of tuple as an \textit{$n$-tuple of rotational shifts}. Some useful properties of this operator tuple are listed in the following proposition.

\begin{propn}\label{properties}
The above defined rotational shifts have the following properties:
\begin{itemize}
    \item [(i)] The tuple $(M_{z_1} R_{[q]_1},\ldots, M_{z_n}R_{[q]_n})\in\T^{q,n}\big(H^2_{\cle}(\D^n)\big)$
    
    \item [(ii)] For each $m=1\ldots,n$, the adjoint of $M_{z_m} R_{[q]_m}$ is given by 
    \begin{equation}\label{adjoint}
    (M_{z_m} R_{[q]_m})^*\Big(\sum_{\bm{k}\in\Z^n}a_{\bm{k}}\z^{\bm{k}}\Big)=\sum_{\bm{k}\in\Z^n}a_{\bm{k}+e_m}\overline{[q]_m}^{\bm{k}}\z^{\bm{k}}
    \end{equation}
    where, $\overline{[q]_m}=\Big(\overline{q(m,1)}^{\frac{1}{2}},\ldots,\overline{q(m,n)}^{\frac{1}{2}}\Big)$ and $e_m=(0,\ldots,0,1,0,\ldots,0)\in\Z^n$ where $1$ appears in the $m$-th position only.
    
    \item [(iii)] For each $m=1\ldots,n$, $M_{z_m} R_{[q]_m}$ is an isometry and for all $f\in H^2_{\cle}(\D^n)$
    \[\big\|\big(M_{z_m} R_{[q]_m}\big)^{*p}(f)\big\|\to 0 \quad\text{as}\quad p\to\infty.\]
    \item [(iv)] The tuple $(M_{z_1} R_{[q]_1},\ldots, M_{z_n}R_{[q]_n})$ is doubly $q$-commuting and hence satisfies Szeg\"o positivity.
\end{itemize}
\end{propn}
 
\begin{proof}
(i)  To prove $q$-commutativity we calculate 
$(M_{z_i}R_{[q]_i})(M_{z_j}R_{[q]_j})$ and $(M_{z_j}R_{[q]_j})(M_{z_i}R_{[q]_i})$ separately, for arbitrary $i,j=1,\ldots,n$.
\begingroup
\allowdisplaybreaks
\begin{align*}
    (M_{z_i}R_{[q]_i})(M_{z_j}R_{[q]_j})\Big(\sum_{\bm k\in\Z^n}a_{\bm k}\z^{\bm k}\Big)=&(M_{z_i}R_{[q]_i})M_{z_j}\Big(\sum_{\bm k\in\Z^n}a_{\bm k}{[q]_j}^{\bm k}\z^{\bm k}\Big)\\
    =&(M_{z_i}R_{[q]_i})\Big(\sum_{\bm k\in\Z^n}a_{\bm k}{[q]_j}^{\bm k}\z^{{\bm k}+e_j}\Big)\\
    =&\sum_{\bm k\in\Z^n}a_{\bm k}[q]_j^{\bm k}[q]_i^{{\bm k}+e_j}\z^{{\bm k}+e_j+e_i}\\
    =&\sum_{\bm k\in\Z^n}a_{\bm k}[q]_j^{\bm k}[q]_i^{\bm k}q(i,j)^{\frac{1}{2}}\z^{{\bm k}+e_j+e_i}.
\end{align*}
\endgroup
Similarly,
\begingroup
\allowdisplaybreaks
\begin{align*}
    (M_{z_j}R_{[q]_j})(M_{z_i}R_{[q]_i})\Big(\sum_{\bm k\in\Z^n}a_{\bm k}\z^{\bm k}\Big)=& (M_{z_j}R_{[q]_j})\Big(\sum_{\bm k\in\Z^n}a_{\bm k}{[q]_i}^{\bm k}\z^{{\bm k}+e_i}\Big)\\
    =&\sum_{\bm k\in\Z^n}a_{\bm k}[q]_i^{\bm k}[q]_j^{{\bm k}+e_i}\z^{{\bm k}+e_i+e_j}\\
    =&\sum_{\bm k\in\Z^n}a_{\bm k}[q]_i^{\bm k}[q]_j^{\bm k}q(j,i)^{\frac{1}{2}}\z^{{\bm k}+e_j+e_i}.
\end{align*}
\endgroup
Since, $q(i,j)=\overline{q(j,i)}$ and $|q(i,j)|=1$, we have from the above
\[
q(j,i)^{\frac{1}{2}}(M_{z_i}R_{[q]_i})(M_{z_j}R_{[q]_j})=q(i,j)^{\frac{1}{2}}(M_{z_j}R_{[q]_j})(M_{z_i}R_{[q]_i}).
\]
Multiplying both sides by $\overline{q(j,i)}^{\frac{1}{2}}$ we get
\[
(M_{z_i}R_{[q]_i})(M_{z_j}R_{[q]_j})=q(i,j)(M_{z_j}R_{[q]_j})(M_{z_i}R_{[q]_i}).
\]
Therefore, $(M_{z_1} R_{[q]_1},\ldots, M_{z_n}R_{[q]_n})\in \T^{q,n}\big(H^2_{\cle}(\D^n)\big)$.

(ii) For each $m=1,\dots,n$,
\begingroup
\allowdisplaybreaks
\begin{align*}
    \Big<(M_{z_m}R_{[q]_m})\Big(\sum_{\bm k\in\Z^n}a_{\bm k}\z^{\bm k}\Big), \sum_{\bm l\in\Z^n}b_{\bm l}\z^{\bm l} \Big>=&\Big<\sum_{\bm k\in\Z^n}a_{\bm k}[q]_m^{\bm k}\z^{{\bm k}+e_m},\sum_{\bm l\in\Z^n}b_{\bm l}\z^{\bm l}\Big>\\
    =&\sum_{\bm k\in\Z^n}[q]_m^{\bm k}\Big<a_{\bm k},b_{{\bm k}+e_m}\Big>\\
    =&\Big<\sum_{\bm k\in\Z^n}a_{\bm k}\z^{{\bm k}},\sum_{\bm l\in\Z^n}b_{{\bm l}+e_m}\overline{[q]_m}^{\bm l}\z^{\bm l}\Big>.
\end{align*}
\endgroup

(iii) For each $m=1,\ldots,n$, rotation operators $R_{[q]_m}$,  are isometries, which is easy to see from (\ref{rotation}), so are the operators $M_{z_m}R_{[q]_m}$. Also, for each $m=1,\dots,n$,
\[
\Big\|(M_{z_m} R_{[q]_m})^{*p}\Big(\sum_{\bm{k}\in\Z^n}a_{\bm{k}}\z^{\bm{k}}\Big)\Big\|^2\leq \sum_{\bm{k}\in\Z^n}\|a_{\bm{k}+pe_m}\|^2\to\infty\quad\text{as}\, p\to\infty.
\]

(iv) For arbitrary $i,j=1,\ldots,n$,
\begingroup
\allowdisplaybreaks
\begin{align*}
    (M_{z_i}R_{[q]_i})(M_{z_j}R_{[q]_j})^*\Big(\sum_{\bm k\in\Z^n}a_{\bm k}\z^{\bm k}\Big)
    =&(M_{z_i}R_{[q]_i})\Big(\sum_{\bm k\in\Z^n}a_{\bm{k}+e_j}\overline{[q]_j}^{\bm{k}}\z^{\bm{k}}\Big)\\
    =& \sum_{\bm k\in\Z^n}a_{\bm{k}+e_j}[q]_i^{\bm{k}}\overline{[q]_j}^{\bm{k}}\z^{{\bm k}+e_i}.
\end{align*}
\endgroup
On the other hand,
\begingroup
\allowdisplaybreaks
\begin{align*}
    (M_{z_j}R_{[q]_j})^*(M_{z_i}R_{[q]_i})\Big(\sum_{\bm k\in\Z^n}a_{\bm k}\z^{\bm k}\Big)
    =&(M_{z_j}R_{[q]_j})^*\Big(\sum_{\bm k\in\Z^n}a_{\bm k}[q]_i^{\bm{k}}\z^{{\bm k}+e_i}\Big)\\
    =& \sum_{\bm k\in\Z^n}a_{{\bm k}+e_j}[q]_i^{\bm{k}}q(i,j)^{\frac{1}{2}}\overline{[q]_j}^{\bm{k}}q(i,j)^{\frac{1}{2}}\z^{{\bm k}+e_i}.
 \end{align*}
 \endgroup
So,

\begin{equation}\label{doubly commuting}
(M_{z_i}R_{[q]_i})(M_{z_j}R_{[q]_j})^*= \overline{q(i,j)}(M_{z_j}R_{[q]_j})^*(M_{z_i}R_{[q]_i}).
\end{equation}
The above relation and part (i) of this proposition together imply that  $(M_{z_1} R_{[q]_1},\ldots, M_{z_n}R_{[q]_n})$ is doubly $q$-commuting. Finally, we conclude from Proposition \ref{doublyszego} that the tuple satisfies Szeg\"o positivity.
\end{proof}
\textbf{Isometric dilations.} For given $q$ as in Definition \ref{qcommute} and for a Hilbert space $\clh$, we choose a tuple $T=(T_1,\ldots, T_n)\in\T^{q,n}(\clh)$. Under certain assumptions we shall find isometric dilation of $T$. Our aim is to dilate $T$ in the following way: We shall find a tuple of isometries $V=(V_1,\ldots, V_n)\in\T^{q,n}(\clk)$ on some Hilbert space $\clk$ and an isometry $\Pi:\clh\to\clk$ such that
\begin{equation}\label{pi}
  \Pi T_m^*=V_m^* \Pi  
\end{equation}
for all $m=1,\ldots,n$. Taking $\clr=\ran \Pi$, the above equation tells us that $(T_1,\ldots,T_n)$ is unitarily equivalent to the tuple $(P_{\clr}V_1|_{\clr},\ldots, P_{\clr}V_n|_{\clr})$. Also,
\begin{equation}\label{co}
T_m^*\cong V_m^*|_{\clr} 
\end{equation}
for all $m=1,\ldots,n$ and hence
\[
T^{*\bm k}\cong V^{*\bm k}|_{\clr},\text{ consequently } T^{\bm k}\cong P_{\clr}V^{\bm k}|_{\clr}
\]
for all $\bm k\in\Z^n$, where for $\bm k=(k_1,\ldots,k_n)\in\Z^n$, $T^{\bm k}=T_1^{k_1}\cdots T_n^{k_n}$ and $T^{*\bm k}=T_n^{*k_n}\cdots T_1^{*k_1}$  (see \cite{BDHS} for more detailed explanation). This shows that $V$ is a dilation of $T$ and the relation $(\ref{co})$ tells us that $V$ is in particular a co-extension of $T$. Note that, in the above notations of $T^{\bm k}$ and $T^{*\bm k}$, maintaining the order of appearance of the operators which involve in the composition, is important and we shall use these notations very often in the rest of the article.

The following theorem is very crucial for our purpose, as well as, after adding some extra assumption in the hypothesis, it provides isometric dilations for certain class of operator tuples which shall be discussed right after the theorem.  
\begin{thm}\label{dilation for szego tuple}
Let $T=(T_1,\ldots,T_n)\in\T^{q,n}(\clh)$ be a tuple which satisfies Szeg\"o positivity. Denote,
\[
D_{T^*}:=\mathbb{S}_n^{-1}(T,T^*)^{\frac{1}{2}} \text{ and } \cld_{T^*}:=\overline{ran}\, \mathbb{S}_n^{-1}(T,T^*).
\]
Then, there is a map $\Pi: \clh\to  H^2_{\cld_{T^*}}(\D^n)$ defined by
\[
\Pi h(\z)= \sum_{\bm k\in\Z^n}\Big(\prod_{1\leq i<j\leq n}q(i,j)^{\frac{k_ik_j}{2}}\Big)D_{T^*}T^{*\bm k}h\z^{\bm k}
\]
such that 
\begin{equation}\label{key eq}
\Pi T_m^*=(M_{z_m}R_{[q]_m})^*\Pi
\end{equation}
for all $m=1,\ldots,n$ and
\[\|\Pi h\|^2=\lim_{l\to\infty}\sum_{F\subset\{1,\ldots,n\}}(-1)^{|F|}\|T_F^l h\|^2\]
for all $h\in\clh$.
\end{thm}
\begin{proof}
First we shall prove the identity (\ref{key eq}). Now for $h\in \mathcal{H}$, we have
\begingroup
\allowdisplaybreaks
\begin{align*}
   (M_{z_m}&R_{[q]_m})^*\Pi h(\z)\\
   =&(M_{z_m}R_{[q]_m})^*\sum_{\bm k\in\Z^n}\Big(\prod_{1\leq i<j\leq n}q(i,j)^{\frac{k_ik_j}{2}}\Big)D_{T^*}T^{*\bm k}h\z^{\bm k}\\
   =& \sum_{\bm k\in\Z^n}\Big(\prod_{\substack{1\leq i<j\leq n\\i,j\neq m}}q(i,j)^{\frac{k_ik_j}{2}}\Big)\Big(\prod_{1\leq i<m}q(i,m)^{\frac{k_i(k_m+1)}{2}}\Big)\Big(\prod_{m<j\leq n}q(m,j)^{\frac{(k_m+1)k_j}{2}} \Big)\\
   &\hspace{1cm}\Big(\prod_{1\leq j\leq n}\overline{q(m,j)}^{\frac{k_j}{2}}\Big)D_{T^*}T^{*{\bm k}+e_m}h\z^{\bm k}\hspace{6.8cm}[\text{by}\, (\ref{adjoint})]\\
   =&\sum_{\bm k\in\Z^n}\Big(\prod_{\substack{1\leq i<j\leq n\\i,j\neq m}}q(i,j)^{\frac{k_ik_j}{2}}\Big)\Big(\prod_{1\leq i<m}q(i,m)^{\frac{k_i(k_m+1)}{2}}\Big)\Big(\prod_{m<j\leq n}q(m,j)^{\frac{k_mk_j}{2}} \Big)\\&\hspace{1cm}\Big(\prod_{1\leq j\leq m}\overline{q(m,j)}^{\frac{k_j}{2}}\Big)D_{T^*}T^{*{\bm k}+e_m}h\z^{\bm k}\hspace{4.3cm}[\text{as }\overline{q(m,j)}q(m,j)=1]\\
   =& \sum_{\bm k\in\Z^n}\Big(\prod_{\substack{1\leq i<j\leq n\\i,j\neq m}}q(i,j)^{\frac{k_ik_j}{2}}\Big)\Big(\prod_{1\leq i<m}q(i,m)^{\frac{k_ik_m}{2}}\Big)\Big(\prod_{m<j\leq n}q(m,j)^{\frac{k_mk_j}{2}}\Big)\\
   &\hspace{1cm}\Big(\prod_{1\leq j<m}{q(j,m)}^{k_j}\Big)D_{T^*}T^{*{\bm k}+e_m}h\z^{\bm k}\\
   =& \sum_{\bm k\in\Z^n}\Big(\prod_{1\leq i<j\leq n}q(i,j)^{\frac{k_ik_j}{2}}\Big)\Big(\prod_{1\leq j<m}q(j,m)^{k_j}\Big)D_{T^*}T^{*{\bm k}+e_m}h\z^{\bm k}\\
   =&\sum_{\bm k\in\Z^n}\Big(\prod_{1\leq i<j\leq n}q(i,j)^{\frac{k_ik_j}{2}}\Big)\Big(\prod_{1\leq j<m}q(j,m)^{k_j}\Big)\Big(\prod_{1\leq j<m}\overline{q(j,m)}^{k_j}\Big)D_{T^*}T^{*{\bm k}}T_m^*h\z^{\bm k}\\
   &\hspace{3.9cm}[\text{since},\,\,T_jT_m=q(j,m)T_mT_j\,\,\, \text{implies}\,\,\, T_m^*T_j^*=\overline{q(j,m)}T_j^*T_m^*\, ]\\
   =&\sum_{\bm k\in\Z^n}\Big(\prod_{1\leq i<j\leq n}q(i,j)^{\frac{k_ik_j}{2}}\Big)D_{T^*}T^{*{\bm k}}T_m^*h\z^{\bm k}\\
   =&\Pi T_m^*h(\z).
\end{align*}
\endgroup
This proves the identity (\ref{key eq}). For the remaining part we follow the technique of \cite{AM}.
\begingroup
\allowdisplaybreaks
\begin{align*}
    \|\Pi h\|^2=&\sum_{\bm k\in\Z^n}\Big\|\Big(\prod_{1\leq i<j\leq n}q(i,j)^{\frac{k_ik_j}{2}}\Big)D_{T^*}T^{*\bm k}h\Big\|\\
    =& \sum_{\bm k\in\Z^n}\|D_{T^*}T^{*\bm k}h\|\\
    =& \lim_{l\to\infty}\sum_{\substack{\bm k\in\Z^n\\\max k_j\leq l-1}}\sum_{F\subset\{1,\ldots,n\}}(-1)^{|F|}\|T_F^*T^{*\bm k}h\|\\
    =& \lim_{l\to\infty}\sum_{\substack{\bm{\beta}\in\Z^n\\\max \beta_j\leq l}}\|T^{*\bm{\beta}}h\|^2\sum_{\substack{\bm k\leq\bm{\beta}\\ \max (\beta_j-k_j)\leq 1\\\max k_j\leq l-1}}(-1)^{|\bm{\beta}|-|\bm k|}.
\end{align*}
\endgroup
Observe that, for all $\bm{\beta}\in\Z^n$,
\[
\sum_{\substack{\bm k\leq\bm{\beta}\\ \max (\beta_j-k_j)\leq 1\\\max k_j\leq l-1}}(-1)^{|\bm{\beta}|-|\bm k|}=0
\]
except the case that $\{\beta_1,\ldots,\beta_n\}\subset\{0,l\}$ and if $\{\beta_1,\ldots,\beta_n\}\subset\{0,l\}$ the sum is equal to $(-1)^{|\text{supp}\, \beta|}$. So, 
\[
\|\Pi h\|^2=\lim_{l\to\infty}\sum_{F\subset\{1,\ldots,n\}}(-1)^{|F|}\|T^{*l}_Fh\|^2,
\]
for all $h\in\clh$.
\end{proof}
A tuple of bounded operators $(T_1,\ldots,T_n)$ on a Hilbert space $\clh$ is said to be \textit{pure} if for all $i=1,\ldots,n$ and for all $h\in\clh$, $\|T_i^{*m}(h)\|\to 0$ as $m\to\infty$. We have seen in Proposition \ref{properties} that $(M_{z_1} R_{[q]_1},\ldots, M_{z_n}R_{[q]_n})\in\T^{q,n}\big(H^2_{\cle}(\D^n)\big)$ is a pure tuple of isometries. In the following theorem, we describe explicit isometric dilations for a class of pure tuples in $\T^{q,n}(\clh)$ which generalizes the dilation result of \cite{CV1} for commutative case and it is a particular case of above theorem.
\begin{thm}\label{pure-dil}
Let $T=(T_1,\ldots,T_n)\in\T^{q,n}(\clh)$ be a pure tuple satisfying Szeg\"o positivity. Let $D_{T^*}$, $\cld_{T^*}$ and $\Pi:\clh\to H^2_{\cld_{T^*}}(\D^n)$ be as in Theorem \ref{dilation for szego tuple}. Then, $\Pi$ is an isometry and \begin{equation}\label{key eq1}
\Pi T_m^*=(M_{z_m}R_{[q]_m})^*\Pi.
\end{equation}
for all $m=1,\ldots,n$. In other words, $(M_{z_1} R_{[q]_1},\ldots, M_{z_n}R_{[q]_n})$ on $H^2_{\cld_{T^*}}(\D^n)$ is an isometric dilation of the $q$-commuting tuple $T=(T_1,\ldots,T_n)$.
\end{thm}
\begin{proof}
Only to prove that $\Pi$ is an isometry as everything else follows from Theorem \ref{dilation for szego tuple}. Observe that,
\[
\|\Pi h\|^2=\lim_{l\to\infty}\sum_{F\subset\{1,\ldots,n\}}(-1)^{|F|}\|T^{*l}_Fh\|^2=\|h\|^2+\lim_{l\to\infty}\sum_{\emptyset\neq F\subset\{1,\ldots,n\}}(-1)^{|F|}\|T^{*l}_Fh\|^2.
\]
Now if $T$ is pure, the last term of the above equality is $0$. Therefore, $\Pi$ is an isometry.
\end{proof}
Generalizing Szeg\"o positivity, we say that $T\in\T^{q,n}(\clh)$ satisfies \textit{Brehmer positivity} if for all $G\subset\{1,\ldots,n\}$
\[
\sum_{F\subset G}(-1)^{|F|}T_FT_F^*\geq 0.
\]
With this we define another class of operator tuple, namely $\mathfrak{B}^{q,n}(\clh)$, as the set of all $T\in\T^{q,n}(\clh)$ such that $T$ satisfies Brehmer positivity, that is,
\[
\mathfrak{B}^{q,n}(\clh)=\Big\{T\in \T^{q,n}(\clh):\sum_{F\subset G}(-1)^{|F|}T_FT_F^*\geq 0,\,\, \text{for all}\,\,G\subset\{1,\ldots,n\}\Big\}.
\]
This class naturally extends the Brehmer class of commuting operator tuples having isometric dilations. We find explicit isometric dilations for the class $\mathfrak{B}^{q,n}(\clh)$ and in order to do that we mostly follow the explicit construction from \cite{AM}. To prove the dilation theorem for this class, we establish several technical lemmas. Before proceeding further we fix some more notations which will be used afterwords. For $T\in \mathfrak{B}^{q,n}(\clh)$ and for a subset $G=\{j_1,\ldots,j_r\}$ of $\{1,\ldots,n\}$, we denote
\[
T(G):=(T_{j_1},\ldots,T_{j_r}),\quad \mathbb{S}_r^{-1}(T(G),T(G)^*)=\sum_{F\subset G}(-1)^{|F|}T_FT_F^*
\]
and
\[
D_{T(G)^*}:=\mathbb{S}_r^{-1}(T(G),T(G)^*)^{\frac{1}{2}}\text{ and } \cld_{T(G)^*}:=\overline{\ran}\,D_{T(G)^*}.
\]
Our first lemma is the $q$-commuting version of Fuglede-Putnam's theorem \cite{Fu} for commuting operators and proof of it is a simple application of the same theorem. The lemma is as follows.
\begin{lemma}\label{fug-put}
If $N$ is normal operator on a Hilbert space $\mathcal{H}$ and $X:\mathcal{H}\to \mathcal{H}$ is an operator such that $XN=qNX$, where $q$ is a complex number, then $XN^*=\overline{q}N^*X$.
\end{lemma}
\begin{proof}
Consider 
\[
A=\begin{bmatrix}
0 & 0\\
X & 0
\end{bmatrix} \text{ and } B=\begin{bmatrix}
N & 0\\
0 & qN
\end{bmatrix}.
\]
Clearly, $B$ is normal operator on $\mathcal{H}\oplus \mathcal{H}$ and $AB=BA$. Then by Fuglede-Putnam's theorem, $AB^*=B^*A$ and hence $XN^*=\overline{q}N^*X$.
\end{proof}
Our next lemma in this sequel is following.
\begin{lemma}\label{coisometry}
Let $G=\{m_1,\ldots,m_r\}\subset\{1,\ldots,n\}$ and $\overline{G}=\{1,\ldots,n\}\smallsetminus G$. Let $T=(T_1,\ldots,T_n)\in\T^{q,n}(\clh)$ be such that $\mathbb{S}_r^{-1}(T(G),T(G)^*)\geq 0$.
Suppose $T_j$ is a co-isometry for all $j\in\overline{G}$. Then there exist a Hilbert space $\clh_G$  and an $n$-tuple of isometries $V=(V_1,\ldots,V_n)$ on $H^2_{\clh_G}(\D^r)$ such that $V$ is doubly $q$-commuting and 
\[
 \Pi_GT_{i}^*=V_i^*\Pi_G\ (1\le i\le n)
\]
where $\Pi_G: \clh\to H^2_{\clh_G}(\D^r)$ is a map such that
\[
\|\Pi_G h\|^2=\lim_{l\to\infty}\sum_{F\subset G}(-1)^{|F|}\|T^{*l}_Fh\|^2,
\]
for all $h\in\clh$. Moreover,
\[
 V_{m_i}=M_{z_{m_i}}R_{[q]_{m_i}} (1\leq i\leq r) \text{ and } V_j=(I\otimes U_j)R_{[q]_j}^2\ (j\in \overline{G})
\]
for some $q$-commuting unitaries $U_j$ on $\clh_G$. (For this lemma we fix that $\bm z=(z_{m_1},\ldots,z_{m_r})\in\D^r$).
\end{lemma}

\begin{proof}
First we observe that, for each $j\in\overline{G}$,
\begin{align*}
    T_jD^2_{T(G)^*}T_j^*=\sum_{F\subset G}(-1)^{|F|}T_jT_FT_F^*T_j^*
    =\sum_{F\subset G}(-1)^{|F|}\Big(\prod_{i\in F}q(j,i)\overline{q(j,i)}\Big)T_FT_jT_j^*T_F^*
    =D^2_{T(G)^*}.
\end{align*}
	So, by Douglas' lemma (see \cite{Douglas}), there exist co-isometries $S_j:\cld_{T(G)^*}\to\cld_{T(G)^*}$ $(j\in\overline{G})$ such that 
	\begin{equation}\label{U_j}
	  	S_j^*D_{T(G)^*}h=D_{T(G)^*}T_j^*h,  
	\end{equation}
for all $h\in\clh$. It follows from the $q$-commutativity of the sub-tuple $T(G)$ and the above equation that $\{S_j: j\in\overline{G}\}$ is a tuple of $q$-commuting co-isometries on $\cld_{T(G)^*}$. Since, due to Theorem 6.1 of \cite{BS}, every tuple of $q$-commuting isometries extends to a tuple of $q$-commuting unitaries, the tuple $\{S_j: j\in\overline{G}\}$ co-extends to a tuple of $q$-commuting unitaries, say, to $\{U_j:j\in \overline{G}\}$ on some Hilbert space $\clh_G$.

On the other hand, since the $r$-tuple of $q$-commuting contractions $\{T_i:i\in G\}$ satisfies Szeg\"o positivity, apply Theorem \ref{dilation for szego tuple} to get an operator $\Pi_G:\clh\to H^2_{\cld_{T(G)^*}}(\D^r)\subset H^2_{\clh_G}(\D^r)$ such that
\[
\Pi_G h(\z)= \sum_{\bm k\in\Z^r}\Big(\prod_{1\leq t<p\leq r}q(m_t,m_p)^{\frac{k_{m_t}k_{m_p}}{2}}\Big)D_{T(G)^*}T(G)^{*\bm k}h\z^{\bm k}
\]
with
\[
\|\Pi_G h\|^2=\lim_{l\to\infty}\sum_{F\subset G}(-1)^{|F|}\|T^{*l}_Fh\|^2,
\]
for all $h\in\clh$, and
\[
	\Pi_GT_{m_i}^*=\big(M_{z_{m_i}}R_{[q]_{m_i}}\big)^*\Pi_G
\]
for all $i=1\ldots,r$. Now, for all $j\in\overline{G}$,
\begingroup
\allowdisplaybreaks
\begin{align*}
    &R_{[q]_j}^{*2}(I\otimes U_j)^*\Pi_G h(\z)\\
    =&R_{[q]_j}^{*2} \sum_{\bm k\in\Z^r}\Big(\prod_{1\leq t<p\leq r}q(m_t,m_p)^{\frac{k_{m_t}k_{m_p}}{2}}\Big)D_{T(G)^*}T_j^*T(G)^{*\bm k}h\z^{\bm k}\hspace{4cm}[\text{using \eqref{U_j}}]\\
    =& R_{[q]_j}^{*2} \sum_{\bm k\in\Z^r}\Big(\prod_{1\leq t<p\leq r}q(m_t,m_p)^{\frac{k_{m_t}k_{m_p}}{2}}\Big)\Big(\prod_{1\leq t\leq r}\overline{q(m_t,j)}^{k_{m_t}}\Big)D_{T(G)^*}T(G)^{*\bm k}T_j^*h\z^{\bm k}\\
    =& \sum_{\bm k\in\Z^r}\Big(\prod_{1\leq t<p\leq r}q(m_t,m_p)^{\frac{k_{m_t}k_{m_p}}{2}}\Big)\Big(\prod_{1\leq t\leq r}\overline{q(m_t,j)}^{k_{m_t}}\Big)\Big(\prod_{1\leq t\leq r}\overline{q(j,m_t)}^{k_{m_t}}\Big)D_{T(G)^*}T(G)^{*\bm k}T_j^*h\z^{\bm k}\\
    =& \sum_{\bm k\in\Z^r}\Big(\prod_{1\leq t<p\leq r}q(m_t,m_p)^{\frac{k_{m_t}k_{m_p}}{2}}\Big)D_{T(G)^*}T(G)^{*\bm k}T_j^*h\z^{\bm k}\\
    =& \Pi_G T_j^* h(\z)
\end{align*}
\endgroup
So,
\[
\Pi_GT_{j}^*=\big((I\otimes U_j)R^2_{[q]_j}\big)^*\Pi_G
\] 
for all $j\in\overline{G}$. Now, for all $i\in G$ and $j\in\overline{G}$,
\begingroup
\allowdisplaybreaks
\begin{align*}
    (M_{z_i}R_{[q]_i})&\big((I\otimes U_j)R^2_{[q]_j}\big)\Big(\sum_{\bm k\in\Z^r}a_{\bm k}\z^{\bm k}\Big)\\
    =&(M_{z_i}R_{[q]_i})\Big(\sum_{\bm k\in\Z^r}U_ja_{\bm k}[q]_j^{2\bm k}\z^{\bm k}\Big)\\
    =&\sum_{\bm k\in\Z^r}U_ja_{\bm k}[q]_j^{2\bm k}[q]_i^{\bm k}\z^{{\bm k}+e_i}\\
    =&\sum_{\bm k\in\Z^r}U_ja_{\bm k}\overline{q(j,i)}[q^2]_j^{{\bm k}+e_i}[q]_i^{\bm k}\z^{{\bm k}+e_i}\\
    =&\big((I\otimes U_j)R^2_{[q]_j}\big)\sum_{\bm k\in\Z^r}a_{\bm k}q(i,j)[q]_i^{\bm k}\z^{{\bm k}+e_i}\\
    =& q(i,j)\big((I\otimes U_j)R^2_{[q]_j}\big)(M_{z_i}R_{[q]_i})\Big(\sum_{\bm k\in\Z^r}a_{\bm k}\z^{\bm k}\Big).
\end{align*}
\endgroup
So, for each $i\in G$ and $j\in\overline{G}$, denoting 
\[
V_i=M_{z_i}R_{[q]_i}\text{ and } V_j=(I\otimes U_j)R^2_{[q]_j}
\]
we see from the above calculation that the pair $(V_i,V_j)$ is $q(i,j)$-commuting. Since $R_{[q]_j}$ is unitary, $V_j$ is also a unitary for all $j\in\overline{G}$. Hence, by Lemma \ref{fug-put}, $(V_i,V_j)$ is doubly $q(i,j)$-commuting. Also, for all $j_1,j_2\in\overline{G}$, 
\begingroup
\allowdisplaybreaks
\begin{align*}
    \big((I\otimes U_{j_1})R^2_{[q]_{j_1}}\big)&\big((I\otimes U_{j_2})R^2_{[q]_{j_2}}\big)\Big(\sum_{\bm k\in\Z^r}a_{\bm k}\z^{\bm k}\Big)\\
    =&\big((I\otimes U_{j_1})R^2_{[q]_{j_1}}\big)\Big(\sum_{\bm k\in\Z^r}U_{j_2}a_{\bm k}[q]_{j_2}^{2\bm k}\z^{\bm k}\Big)\\
    =&\sum_{\bm k\in\Z^r}U_{j_1}U_{j_2}a_{\bm k}[q]_{j_1}^{2\bm k}[q]_{j_2}^{2\bm k}\z^{\bm k}\\
    =& q(j_1,j_2)\sum_{\bm k\in\Z^r}U_{j_2}U_{j_1}a_{\bm k}[q]_{j_1}^{2\bm k}[q]_{j_2}^{2\bm k}\z^{\bm k}\\
    =& q(j_1,j_2)\big((I\otimes U_{j_1})R^2_{[q]_{j_1}}\big)\big((I\otimes U_{j_2})R^2_{[q]_{j_2}}\big)\Big(\sum_{\bm k\in\Z^r}a_{\bm k}\z^{\bm k}\Big).
\end{align*}
\endgroup
So, again by Lemma \ref{fug-put}, $(V_{j_1},V_{j_2})$ is doubly $q(j_1,j_2)$-commuting for all $j_1,j_2\in\overline{G}$. Also, it follows from Proposition \ref{properties} that $(V_{i_1},V_{i_2})$ is doubly $q(i_1,i_2)$-commuting, for all $i_1,i_2\in G$. Combining all these we conclude that the tuple $(V_1,\ldots,V_n)$ is doubly $q$-commuting.
\end{proof}

We need one more lemma which describes a canonical way to construct co-isometries out of $q$-commuting contractions. 
\begin{lemma}\label{tilde}
Let $T=(T_1,\ldots,T_n)\in\mathfrak{B}^{q,n}(\clh)$. Let $G\subset\{1,\ldots,n\}$ and $\overline{G}=\{1,\dots,n\}\smallsetminus G$. Then there exist a positive operator $X:\clh \to\clh $ and contractions $\tilde{T}_j:\overline{\ran}\,X\to\overline{\ran}\,X$ $(1\le j\le n) $ defined by 
	\[
	\tilde{T}_j^*Xh=X{T}_j^*h,\quad(h\in\clh)
	\]
such that $\tilde{T}_j$ is a co-isometry for all $j\in\overline{G}$.
\end{lemma}
\begin{proof}
Since $T$ is a tuple of contractions, strong operator limit (SOT) of $T_{\overline{G}}^{m}T_{\overline{G}}^{*m}$ exists as $m\to\infty$.
Set
	\[	X^2:=\text{SOT-}\lim_{m\to\infty}T_{\overline{G}}^{m}T_{\overline{G}}^{*m}.
	\]
Since $T_j$ is a contraction, it is easy to see that 
	\[
	T_jX^2T_j^*= \text{SOT-}\lim_{m\to\infty}\Big(\prod_{i\in \overline{G}}q(j,i)^{\frac{m}{2}}\overline{q(j,i)}^{\frac{m}{2}}\Big)T_{\overline{G}}^{m}T_jT_j^*T_{\overline{G}}^{*m} \leq X^2,
	\]
 for all $j=1\ldots,n$. Consequently by Douglas' lemma, for each $1\le j\le n$, there exists a contraction  $\tilde{T}_j:\overline{\mbox{ran}}\,X\to\overline{\ran}\,X$ such that
	\[
	\tilde{T}_j^*Xh=X{T}_j^*h \quad(h\in\clh).
	\]
	Moreover, for all $j\in\overline{G}$, 
	\begin{align*}
	X^2\geq T_jX^2T_j^*=&\text{SOT-}\lim_{m\to\infty}\Big(\prod_{i\in \overline{G}}q(j,i)^{\frac{m}{2}}\overline{q(j,i)}^{\frac{m}{2}}\Big)T_{\overline{G}}^{m}T_jT_j^*T_{\overline{G}}^{*m}\\
	\geq& \text{SOT-}\lim_{m\to\infty}T_{\overline{G}}^{(m+1)}T_{\overline{G}}^{*(m+1)}=X^2,
	\end{align*}
	that is,
	\[
	T_jX^2T_j^*=X^2.
	\]
Hence $\tilde{T}_j$ is a co-isometry for all $j\in\overline{G}$. This completes the proof.
\end{proof}

Combining the above lemmas we now find co-extensions for 
$n$-tuples in $\mathfrak{B}^{q,n}(\clh)$.
\begin{thm}\label{Bre-dil}
Let $T=(T_1,\dots,T_n)\in\mathfrak{B}^{q,n}(\clh)$. Then, for each $G\subset\{1,\ldots,n\}$, there exist a Hilbert space $\clh_G$, doubly $q$-commuting isometries $V_{G,1},\ldots,V_{G,n}$ on $H^2_{\clh_G}(\D^{|G|})$ 
 and an isometry $\Pi:\clh\to\bigoplus_{G}H^2_{\clh_G}(\D^{|G|})$ such that 
	\[
	\Pi T_j^*=\big(\bigoplus_{G}V_{G,j}^*\big)\Pi, \ (1\le j\le n)
	\]
where, by convention $H^2_{\clh_{\emptyset}}(\D^{|\emptyset|}):=\clh_{\emptyset}$. Moreover, for $i\in G$, $V_{G,i}$'s are the rotational shifts, and for $j\in\overline{G}$, $V_{G,j}$'s are some unitaries on the Hardy space $H^2_{\clh_G}(\D^{|G|})$.
\end{thm}
\begin{proof}
 Let us fix $G=\{m_1,\ldots,m_{|G|}\}\subset\{1,\ldots,n\}$ and set $\overline{G}=\{1,\ldots,n\}\smallsetminus G$. Let $X_{G}:\clh\to\clh$ be the positive operator defined by
	\[
	X_{G}^2:=\text{SOT-}\lim_{m\to\infty}T_{\overline{G}}^{m}T_{\overline{G}}^{*m}.
	\]
Then, by Lemma~\ref{tilde}, we have a contraction
 $S_j:\overline{\ran}\,X_G \to \overline{\ran}\,X_G$ defined by 
\[
S_j^*X_Gh=X_G{T}_j^*h,\quad(h\in\clh)
\]
for all $j=1,\dots,n$ so that $S_j$ is a co-isometry 
for all $j\in\overline{G}$. Since, $(T_1,\ldots,T_n)$ is $q$-commuting, it is clear from the definition that $S:=(S_1,\ldots,S_n)$ is also a $q$-commuting tuple. Also, for all $h\in \clh$
\begingroup
\allowdisplaybreaks
\begin{align*}
    \Big\langle\mathbb{S}_{|G|}^{-1}(S(G),S(G)^*)& X_G h,X_G h \Big\rangle \\
    =&\left\langle\sum_{F\subset G}(-1)^{|F|}T_FX_GX_G^*T_F^*h,h\right\rangle\\
    =&\left\langle\sum_{F\subset G}(-1)^{|F|}T_F\Big(\text{SOT-}\lim_{m\to\infty}T_{\overline{G}}^{m}T_{\overline{G}}^{*m}\Big)T_F^*h,h\right\rangle\\
    =&\left\langle\sum_{F\subset G}(-1)^{|F|}\Big(\prod_{j\in\overline{G}, i\in G}q(i,j)^{\frac{m}{2}}\overline{q(i,j)}^{\frac{m}{2}}\Big)\text{SOT-}\lim_{m\to\infty}T_{\overline{G}}^{m}T_FT_F^*T_{\overline{G}}^{*m}h,h\right\rangle\\
    =&\lim_{m\to\infty}\left\langle\Big(\sum_{F\subset G}(-1)^{|F|}T_FT_F^*\Big)T_{\overline{G}}^{*m}h,T_{\overline{G}}^{*m}h\right\rangle\geq 0.
    \end{align*}
    \endgroup
This implies, $\mathbb{S}_{|G|}^{-1}(S(G),S(G)^*)\geq 0$. Apply Lemma~\ref{coisometry} to get a Hilbert space $\clh_G$, a map $\tilde{\Pi}_G:\overline{\mbox{ran}}\,X_G\to H^2_{\clh_G}(\D^{|G|})$ and $n$-tuple of doubly $q$-commuting isometries
	$V=(V_{G,1},\ldots,V_{G,n})$ on $H^2_{\clh_{G}}(\D^{|G|})$ such that
	 \[
\|\tilde{\Pi}_G h\|^2=\lim_{l\to\infty}\sum_{F\subset G}(-1)^{|F|}\|S^{*l}_Fh\|^2,
\]
for all $h\in\clh$, and
	\[
\tilde{\Pi}_GS_i^*=V_{G,i}^*\tilde{\Pi}_G,\
	(1\le i\le n)
	\]
where \[
V_{G,m_i}=M_{z_{m_i}}R_{[q]_{m_i}} (1\leq i\leq |G|)\ \text{and } V_j=(I\otimes U_j)R_{[q^2]_j}\ (j\in \overline{G}).
\]
Now, using the identity $S_j^*X_G=X_GT_j^*$ and setting \[\Pi_G:=\tilde{\Pi}_GX_G,\] we have that 
	\begin{equation}\label{int1}
	\Pi_GT_i^*=V_{G,i}^*\Pi_G,\
	(1\le i\le n)
	\end{equation}
where $\Pi_G :\clh \to H^2_{\clh_G}(\D^{|G|})$ 
is a contraction with 
\begin{align*}
	\|\Pi_Gh\|^2 &=\lim_{l\to\infty}\sum_{F\subset G}(-1)^{|F|}\|S^{*l}_FX_Gh\|^2=\lim_{l\to\infty}\sum_{F\subset G}(-1)^{|F|}\|X_GT^{*l}_Fh\|^2,
\end{align*}
for all $h\in\clh$. Take $\Pi=\bigoplus_G\Pi_G$.
Then it follows from (\ref{int1}) that 
\[
\Pi T_j^*=(\bigoplus_{G}V_{G,j}^* )\Pi
\]
for all $j=1,\dots,n$. To show $\Pi$ is an isometry, for any $h\in\clh$, we compute
\begingroup
\allowdisplaybreaks
\begin{align*}
	\|\Pi h\|^2=&\sum_{G\subset\{1,\ldots,n\}}\|\Pi_G h\|^2\\
	=&\sum_{G\subset\{1,\ldots,n\}}\lim_{l\to\infty}\sum_{F\subset G}(-1)^{|F|}\|X_GT^{*l}_Fh\|^2\\
	=&\sum_{G\subset\{1,\ldots,n\}}\lim_{l\to\infty}\sum_{F\subset G}(-1)^{|F|}\|T_{\overline{G}}^{*l}T^{*l}_Fh\|^2\\
	=&\lim_{l\to\infty}\sum_{E\subset\{1,\ldots,n\}}\|T^{*l}_Eh\|^2\sum_{F\subset E}(-1)^{|F|}.
	\end{align*}
	\endgroup
	If $E\neq\emptyset$ one can see that $\sum_{F\subset E}(-1)^{|F|}=0$. So $\|\Pi h\|^2=\|h\|^2$ for all $h\in\clh$ and hence $\Pi$ is an isometry. The moreover part is now clear from the construction 
of $V_{G,j}$'s. This completes the proof.
\end{proof}

\begin{remark}
\begin{enumerate}
\item In the above theorem, the dilating tuple $V=\Big(\bigoplus_{G}V_{G,1},\ldots,\bigoplus_{G}V_{G,n}\Big)$ is doubly $q$-commuting so are all the sub-tuples of $V$. Proposition \ref{doublyszego} tells us that $V$ satisfy Brehmer positivity and hence $V$ belongs to the class $\mathfrak{B}^{q,n}\Big(\bigoplus_{G}H^2_{\clh_G}\big(\D^{|G|}\big)\Big)$.
    
    \item The class $\mathfrak{B}^{q,n}(\clh)$ enlarges the class discussed in Theorem \ref{pure-dil} and this is for the following reason: Given a pure tuple $T=(T_1,\ldots,T_n)\in\T^{q,n}(\clh)$, a subset $G\subsetneq \{1,\ldots,n\}$, and $m\in\overline{G}$, observe that
    \begin{align*}
    \quad\,\,\,\mathbb{S}_{|G|}^{-1}(T(G),T(G)^*)=&\sum_{p=0}^{\infty}T_m^p\Big(\mathbb{S}_{|G|}^{-1}(T(G),T(G)^*)-T_m\big(\mathbb{S}_{|G|}^{-1}(T(G),T(G)^*)\big)T_m^*\Big)T_m^{*p}\\
    =&\sum_{p=0}^{\infty}T_m^p\Big(\mathbb{S}_{|G|+1}^{-1}\big(T(G\cup\{m\}),T(G\cup\{m\})^*\big)\Big)T_m^{*p}.
   \end{align*}
This shows, by induction, that if $T\in\T^{q,n}(\clh)$ is a pure tuple satisfying Szeg\"o positivity, every sub-tuple of $T$ satisfies Szeg\"o positivity and hence $T\in\mathfrak{B}^{q,n}(\clh)$. On the other hand, if $T$ is also a pure tuple in $\mathfrak{B}^{q,n}(\clh)$ then the positive operator $X_G$, defined in the proof of Theorem \ref{Bre-dil}, is $0$ for all $G\subsetneq \{1,\dots,n\}$ and 
 $X_{G}=I_{\clh}$ for $G=\{1,\dots,n\}$. This implies $\Pi_G=0$
 for all $G\subsetneq\{1,\dots,n\}$, and for $G=\{1,\dots,n\}$, $\Pi_G$ is an isometry and $\clh_G=\cld_{T^*}$. Then it follows from Theorem \ref{Bre-dil} that the $q$-commuting tuple of isometries 
 $ (M_{z_1}R_{[q_1]},\ldots, M_{z_n}R_{[q_n]}) $ on 
 $H^2_{\cld_{T^*}}(\D^n)$ is a dilation of $T$. Thus isometric dilation for the larger class in Theorem \ref{Bre-dil} is a generalization the dilation result of Theorem \ref{pure-dil}.
\end{enumerate}
\end{remark}


\vspace{0.1in} \noindent\textbf{Acknowledgement:} First author's research was carried out at the Ben-Gurion University and the research is partially supported by the Israel Science Foundation grant number 2196-20. The second named author has been supported by the institute post-doctoral fellowship of HRI, Prayagraj.


\bibliographystyle{plain}

\end{document}